\documentclass[11pt]{article}
\usepackage[a4paper, left=1.5cm, right=1.5cm, top=2.5cm, bottom=2.5cm, headsep=1.2cm]{geometry}               
\usepackage[parfill]{parskip}   
\usepackage{graphicx}
\usepackage{amssymb}
\usepackage{color}
\usepackage{amsfonts}
\usepackage{amsmath} 
\usepackage{amsthm}
\usepackage{empheq}
\usepackage{esint}
\usepackage{epstopdf}
\usepackage{verbatim}
\usepackage{ifthen}
\usepackage{comment}
\usepackage{hyperref}
\usepackage{float}

\begin{document}

\title{A model of morphogen transport\\ in the presence of glypicans III}

\author{Marcin Ma\l ogrosz
\footnote{Institute of Applied Mathematics and Mechanics, University of Warsaw, Banacha 2, 02-097 Warsaw, Poland \newline (malogrosz@mimuw.edu.pl)}}
\date{}

\newtheorem{defi}{Definition}
\newtheorem{theo}{Theorem}
\newtheorem{lem}{Lemma}
\newtheorem{rem}{Remark}
\newcommand{\ov}[1]{\overline{#1}}
\newcommand{\un}[1]{\underline{#1}}
\newcommand{\bsym}[1]{\boldsymbol{#1}}
\newcommand{\n}[1]{\lVert#1\rVert}
\newcommand{\eq}[1]{\begin{align}#1\end{align}}

\providecommand{\bs}{\begin{subequations}}
\providecommand{\es}{\end{subequations}}

\maketitle

\begin{abstract}
\noindent
We analyse a stationary problem for the two dimensional model of morphogen transport introduced by Hufnagel et al. The model consists of one linear elliptic PDE posed on $(-1,1)\times(0,h)$  which is coupled via a nonlinear boundary condition with a nonlinear elliptic PDE posed on $(-1,1)\times\{0\}$. The main result is that the system has a unique steady state for all ranges of parameters present in the system. Moreover we consider the problem of the dimension reduction. After introducing an appropriate scaling in the model we prove that, as $h\to0$, the stationary solution converges to the unique steady state of the one dimensional simplification of the model which was analysed in the first part of the paper. The main difficulty in obtaining appropriate estimates stems from the presence of a measure source term in the boundary condition. 
\end{abstract}



\section{Introduction}

Morphogens are signalling molecules which govern the process of cell differentiation in living organisms. They spread from a source spatially localised in the tissue and after a certain amount of time form stable gradients of concentrations. Then receptors located on the surfaces of the cells detect levels of those concentrations and through intracellular pathways this information is conveyed to the nuclei, where the process of gene expression is initiated (see \cite{Wol}). 

The exact mechanism of morphogen transport is still discussed in the literature (see \cite{KPBKBJG-G}, \cite{KW}, \cite{LNW} for modelling and \cite{KLW1}, \cite{KLW2}, \cite{Mal1}, \cite{STW}, \cite{Tel} for mathematical analysis). A model proposed in \cite{Huf} accounts for the transport of morphogen Wingless (Wg) in the imaginal wing disc of the Drosophila Melanogaster. The present paper is the third part of a series of papers where we analyse mathematical properties of this model, which we call \textbf{[HKCS]}. Model \textbf{[HKCS]} has two counterparts - two and one dimensional (denoted respectively \textbf{[HKCS].2D} and \textbf{[HKCS].1D}), depending on the dimension of the domain representing the imaginal wing disc. The main goal of our analysis is a rigourous justification of the so called dimension reduction - \textbf{[HKCS].1D} can be obtained from \textbf{[HKCS].2D} due to shrinking of the rectangular domain in the direction which corresponds to the thickness of the wing disc. 

Model \textbf{[HKCS].2D} accounts for the movement of morphogen molecules by (linear) diffusion in the whole domain $\Omega_h=(-1,1)\times(0,h)$, where $h<<1$ denotes thickness of the disc, while being secreted from a point source localised at $x=0$ on part of the boundary of the wing disc - $\partial_1\Omega_h=(-1,1)\times\{0\}$. Moreover  association-dissociation reactions of morphogen with receptors and glypicans localised on $\partial_1\Omega_h$ are taken under consideration. After association of morphogen with a receptor (resp. glypican) a morphogen-receptor (resp. -glypican) complex is being formed. Apart of the association-dissociation mechanism glypicans also pass among themselves morphogen molecules, which is realised by introducing diffusion of morphogen-glypican complexes along $\partial_1\Omega_h$. Finally morphogen-glypican complexes can further associate with free receptors creating a triple morphogen-glypican-receptor complexes (which are immotile, similarly as morphogen-receptor complexes). Model \textbf{[HKCS].1D} accounts for the same set of reactions between morphogen, glypican and complexes as \textbf{[HKCS].2D}. However it is assumed  that the imaginal wing disc is completely flat ($h=0$) so that the whole dynamics takes place only on $\partial_1\Omega_h$.

In \cite{Mal2} - the first part of our study, we proved that \textbf{[HKCS].1D} is globally well-posed and has a unique steady state.  Article \cite{Mal3} is devoted to the analysis of the evolutionary problem \textbf{[HKCS].2D}. Apart of proving global well-posedness we showed that time dependent solutions of \textbf{[HKCS].2D}, when properly normalised, converge as $h\to0$ to solutions of \textbf{[HKCS].1D}. In this paper we turn our attention to the stationary problem associated with \textbf{[HKCS].2D}. We prove that there is a unique steady state which converges to the equilibrium of \textbf{[HKCS].1D} as $h\to0$. We illustrate our result by performing numerical computations which show that the graph of the stationary solution to \textbf{[HKCS].2D} becomes homogeneous in $x_2$ direction as $h\to0$.
It is worth underlyining that all our results are proved without imposing any artifficial conditions on the parameters which are present in the system.
\\

\subsection{The [HKCS].2D model.}

In this section we recall model \textbf{[HKCS].2D} in a nondimensional form. The model was described in full detail and analysed for the evolutionary case in \cite{Mal3}. For the presentation and analysis of \textbf{[HKCS].1D} - a one dimensional simplification we refer to \cite{Mal2}. \textbf{[HKCS].2D} in a nondimensional form reads:

\bs\label{H2S:System}
\eq{
\partial_t u_{1}^h+div(J_h(u_{1}^h))&=-b_1u_{1}^h,&& (t,x)\in\Omega_{T}\label{H2S:SystemA}\\
\partial_t u_{2}^h-d\partial^2_{x_1} u_{2}^h&=c_1u_{1}^h-(b_2+c_2+c_3u_{3}^h)u_{2}^h+c_5u_{5}^h,&&(t,x)\in(\partial_1\Omega)_{T}\label{H2S:SystemB}\\
\partial_t u_{3}^h&=-(b_3+u_{1}^h+c_3u_{2}^h)u_{3}^h+c_4u_{4}^h+c_5u_{5}^h+p_3,&&(t,x)\in(\partial_1\Omega)_{T}\label{H2S:SystemC}\\
\partial_t u_{4}^h&=u_{1}^hu_{3}^h-(b_4+c_4)u_{4}^h,&&(t,x)\in(\partial_1\Omega)_{T}\label{H2S:SystemD}\\
\partial_t u_{5}^h&=c_3u_{2}^hu_{3}^h-(b_5+c_5)u_{5}^h,&&(t,x)\in(\partial_1\Omega)_{T}\label{H2S:SystemE}
}
\es
with boundary and initial conditions
\begin{align*}
-J_h(u_{1}^h)\nu&=0,&&(t,x)\in(\partial_0\Omega)_{T}\\
-J_h(u_{1}^h)\nu&=-(c_1+u_{3}^h)u_{1}^h+c_2u_{2}^h+c_4u_{4}^h+p_1\delta,&&(t,x)\in(\partial_1\Omega)_{T}\\
\partial_{x_1} u_{2}^h&=0,&&(t,x)\in(\partial\partial_1\Omega)_{T}\\
\boldsymbol{u}^h(0,\cdot)&=\boldsymbol{u}_0,
\end{align*}
where 
\begin{itemize}
\item $\Omega=(-1,1)\times(0,1),\ \partial\Omega=\partial_0\Omega\cup\partial_1\Omega, \ \partial_1\Omega=(-1,1)\times\{0\}, \ 0<T\leq\infty, \ \Omega_T=(0,T)\times \Omega$,
\item  $h>0$ corresponds to the thickness of the wing disc, $J_h(u)=-(\partial_{x_1}u,h^{-2}\partial_{x_2}u)$,
\item  $\nu$ denotes the outer normal unit vector to $\partial\Omega$,
\item $\delta$ denotes a one dimensional Dirac Delta i.e $\delta(\phi)=\phi(0)$ for any $\phi\in C([-1,1])$.
\end{itemize}

In \eqref{H2S:System} $u_1, u_2,u_3,u_4$ and $u_5$ denote concentrations of free morphogen, morphogen-glypican complexes, free receptors, morphogen-receptor complexes and morphogen-glypican-receptor complexes; $c_i$'s are rates of reactions between $u_i$'s; $b_1,b_2$ denote rates of degradations of $u_1,u_2$ while $b_3,b_4,b_5$ are rates of internalisation of $u_3,u_4$ and $u_5$. Finally $d$ is the diffusion rate of $u_2$ along $\partial_{1}\Omega$ while $p_1,p_3$ denote rates of production of morphogen and free receptors.  
Notice the dependence of $J_h(u)$ on parameter $h$.  
\\

From now on we impose the following natural assumptions on the signs of constant parameters
\begin{align}
d, \boldsymbol{b}>0,\  \boldsymbol{c},\boldsymbol{p}\geq 0,\label{H2S:ass}
\end{align}
where $\boldsymbol{b}=(b_1,\ldots,b_5)$ and similarly for $\boldsymbol{c},\boldsymbol{p}$.

\subsection{Notation}

In the whole article $\Omega=(-1,1)\times(0,1)$ and $I=(-1,1)$ are fixed domains. If $U$ is an open subset of $\mathbb{R}^n$  and $1\leq p\leq \infty, \ s\in\mathbb{R}$, we denote by $W^{s}_p(U)$ the fractional Sobolev (also known as Sobolev-Slobodecki) spaces and by $\mathcal{M}(U)$  the Banach space of finite, signed Radon measures on $\overline{U}$  equipped with the total variation norm - $\n{\cdot}_{TV}$.\\

In various estimates we will use a generic constant $C$ which may take different values even in the same paragraph. Constant $C$ may depend on various parameters, but it will never depend on $h$. \\
If $X$ is a normed space we denote by $X^*$ its topological dual. Furthermore if $x\in X$ and $x^*\in X^*$ we denote by $\Big<x^*,x\Big>_{(X^*,X)}=x^*(x)$ a natural pairing between $X$ and its dual. If $H$ is a Hilbert space we denote by $(.|.)_H$ its scalar product. In particular $(\bsym{x}|\bsym{y})_{\mathbb{R}^n}=\sum_{i=1}^nx_iy_i$ and $(f|g)_{L_2(U)}=\int_{U}fg$. To get more familiar with the notation observe that, due to Riesz theorem, for any $x^*\in H^*$ there exists a unique $x\in H$ such that $\Big<x^*,y\Big>_{(H^*,H)}=(x|y)_H$ for all $y\in H$.

\section{Results}

Let us observe that due to the presence of three ODE's in the system \eqref{H2S:System}, the stationary problem may be reduced to a system of two elliptic equations:
\bs\label{H2S:SS2D}
\eq{
div(J_h(u_1))+b_1u_1&=0,&& x\in\Omega\label{H2S:SS2DA}\\
-d\partial^2_{x_1}u_2-c_1u_1+(b_2+c_2+k_2H(u_1,u_2))u_2&=0,&& x\in\partial_1\Omega\label{H2S:SS2DB}
}
\es
with boundary conditions
\bs\label{H2S:SS2Dbcd}
\eq{
-J_h(u_1)\nu&=0,&&x\in\partial_0\Omega\\
-J_h(u_1)\nu&=-(c_1+k_1H(u_1,u_2))u_1+c_2u_2+p_1\delta,&&x\in\partial_1\Omega\\
\partial_{x_1} u_2&=0,&&x\in\partial\partial_1\Omega,
}
\es
where
\begin{align}
k_1=b_4/(b_4+c_4), \ k_2=c_{3}b_5/(b_5+c_5), \ H(u_1,u_2)=p_3/(k_1u_1+k_2u_2+b_3)\label{H2S:Hdef}
\end{align}
and
\begin{align*}
u_3=H(u_1,u_2), \ u_4=\frac{k_1}{b_4}u_1H(u_1,u_2), \ u_5=\frac{k_2}{b_5}u_2H(u_1,u_2).
\end{align*}


We will prove the following two theorems
\newline
\begin{theo}\label{H2S:SStheo1}
For every $h\in(0,1]$ system \eqref{H2S:SS2D}-\eqref{H2S:SS2Dbcd} has a unique nonnegative $W^1_1$ solution $(u_{1}^{h},u_{2}^{h})$ i.e. there exists a unique nonnegative $(u_{1}^{h},u_{2}^{h})\in W^1_1(\Omega)\times W^1_1(\partial_1\Omega)$ such that for every $(v_1,v_2)\in W^1_{\infty}(\Omega)\times W^1_{\infty}(\partial_1\Omega)$ 
\bs\label{H2S:weak}
\eq{
-\int_{\Omega}[J_h(u_1)\nabla v_1+b_1u_1v_1]&=p_1v_1(0)+\int_{\partial_1\Omega}[-(c_1+k_1H(u_1,u_2))u_1+c_2u_2]v_1,\\
\int_{\partial_1\Omega}d\partial_{x_1}u_2\partial_{x_1}v_2&=\int_{\partial_1\Omega}[c_1u_1-(b_2+c_2+k_2H(u_1,u_2))u_2]v_2.
}
\es
Moreover $(u_{1}^{h},u_{2}^{h})\in W^1_p(\Omega)\times W^2_q(\partial_1\Omega)$ for every $1\leq p<2, 1\leq q<\infty$  and
\begin{align}
\n{u_{1}^{h}}_{W^1_p(\Omega)}+h^{-1}\n{\partial_{x_2}u_{1}^{h}}_{L_p(\Omega)}+\n{u_{2}^{h}}_{W^2_q(\partial_1\Omega)}\leq C,\label{H2S:estimate}
\end{align}
where $C$ does not depend on $h$.
\end{theo}
\

\begin{theo}\label{H2S:SStheo2}
Let $(u_{1}^{h},u_{2}^{h})$ be the unique solution of system \eqref{H2S:SS2D}-\eqref{H2S:SS2Dbcd}. Then for every $1\leq p<2, 1\leq q<\infty$ we have the following weak convergence as $h\to0^+$
\bs\label{H2S:conv}
\eq{
u_{1}^{h}&\rightharpoonup u_{1}^{0}\quad {\rm in} \quad W^1_p(\Omega),\\
u_{2}^{h}&\rightharpoonup u_{2}^{0}\quad {\rm in} \quad W^2_q(\partial_1\Omega).
}
\es
Moreover $\partial_{x_2}u_{1}^{0}=0$ (so that $u_{1}^{0}$ depends only on $x_1$) and $(u_{1}^{0},u_{2}^{0})\in W^1_{\infty}(I)\times C^2(\ov{I})$ is the unique solution of
\bs\label{H2S:SS1D}
\eq{
-{u_1}''+(b_1+c_1+k_1H(u_1,u_2))u_1-c_2u_2&=p_1\delta,&&x\in I\label{H2S:SS1DA}\\
-d{u_2}''-c_1u_1+(b_2+c_2+k_2H(u_1,u_2))u_2&=0,&&x\in I\label{H2S:SS1DB}\\
{u_1}'&={u_2}'=0,&&x\in\partial I.
}
\es
\end{theo}

\begin{rem}
Notice that \eqref{H2S:SS1D} is the stationary problem associated with model \textbf{[HKCS].1D} (analysed previously in \cite{Mal2}). Thus Theorem \eqref{H2S:SStheo2} is the rigourous formulation of the dimension reduction of the model \textbf{[HKCS].2D} in the stationary case.
\end{rem}

On Figure \eqref{H2S:wykres} we present graphs of $u_{1}^{h}$ for several values of $h$.  Notice that as $h$ becomes smaller the graph of $u_{1}^{h}$ becomes homogeneous in the $x_2$ direction. 
\\
\

\section{Solvability of certain linear systems with measure valued sources}

To prove Theorem \ref{H2S:SStheo1} we will use two lemmas concerning solvability of linear elliptic boundary value problems with low regularity data.
\\

\begin{lem}\label{H2S:bvp1lem}
Assume that $0\leq a_0\in L_{\infty}(\Omega), \ 0\leq a_{11}\in L_{\infty}(\partial_1\Omega)$. Then for every $h\in(0,1], \ \lambda>0$ and $\mu_{\Omega}\in\mathcal{M}(\Omega), \ \mu_I\in\mathcal{M}(I)$ the following boundary value problem
\bs\label{H2S:bvp1}
\eq{
div(J_h(u))+(\lambda+a_0)u&=\mu_{\Omega}, && x\in \Omega\\
-J_h(u)\nu&=0, &&x\in\partial_0\Omega\\
-J_h(u)\nu+a_{11}u&=\mu_{I},&& x\in\partial_1\Omega
}
\es
has a unique $W^1_1$ solution i.e. there exists a unique $u\in W^1_1(\Omega)$ such that for every $v\in W^1_{\infty}(\Omega)$ 
\begin{align}
\int_{\Omega}[-J_{h}(u)\nabla v +(\lambda +a_0)uv]+\int_{\partial_1\Omega}a_{11}uv&=\int_{\Omega}vd\mu_{\Omega}+\int_{\partial_1\Omega}vd\mu_{I}.\label{H2S:weaklem1}
\end{align}
Moreover $u\in W^1_p(\Omega)$ for every $p<2$ and
\begin{align}\label{H2S:bvp1est}
\n{u}_{W^1_p(\Omega)}+h^{-1}\n{\partial_{x_2}u}_{L_p(\Omega)}\leq C(\n{\mu_{\Omega}}_{TV}+\n{\mu_I}_{TV}),
\end{align}
where $C$ depends only on $p,\lambda,\n{a_0}_{L_{\infty}(\Omega)},\n{a_{11}}_{L_{\infty}(\partial_1\Omega)}$.
If $\mu_{\Omega}, \mu_{I}\geq 0$ then $u\geq 0$.
\end{lem}

\
\\

\begin{figure}[H]
\begin{center}
\includegraphics[height=1.5in, trim=65 0 0 50]{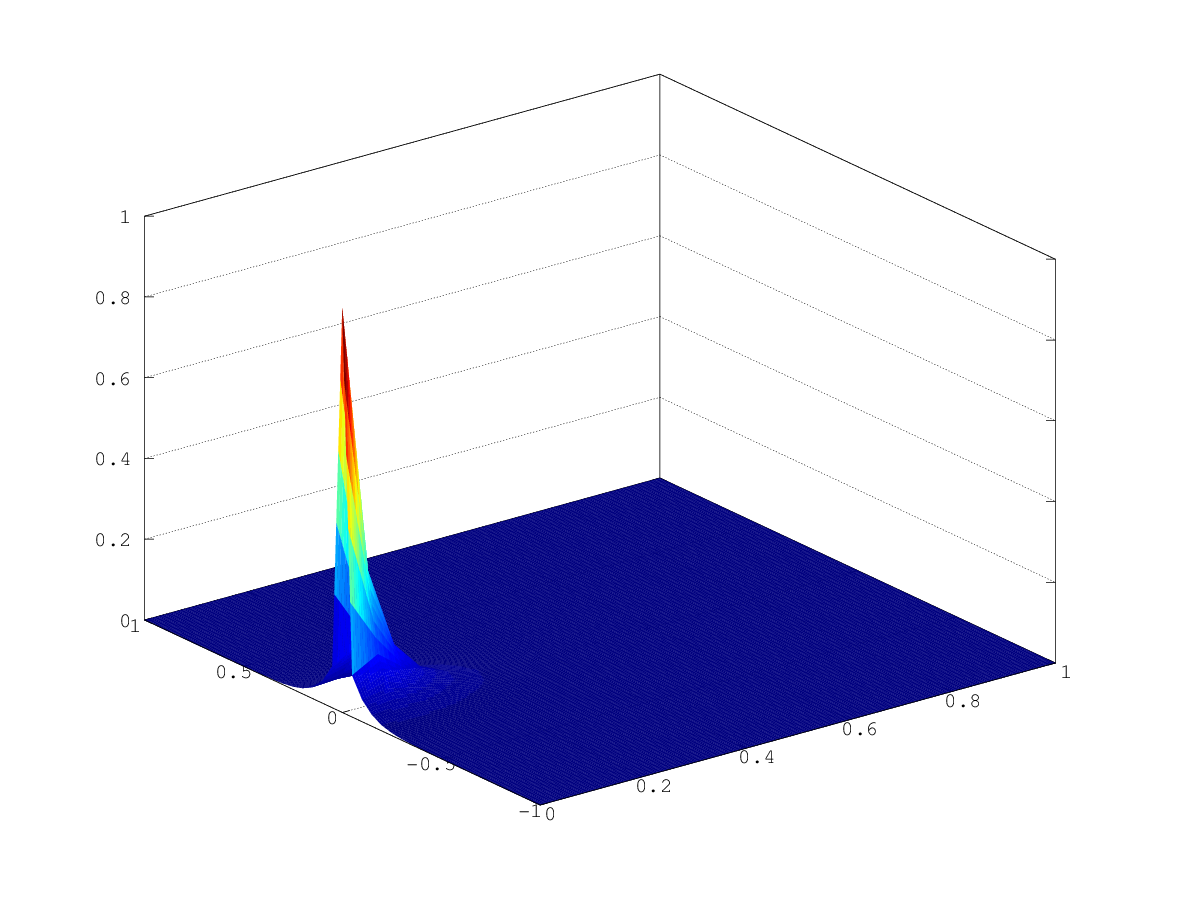}
\includegraphics[height=1.5in, trim=50 0 0 50]{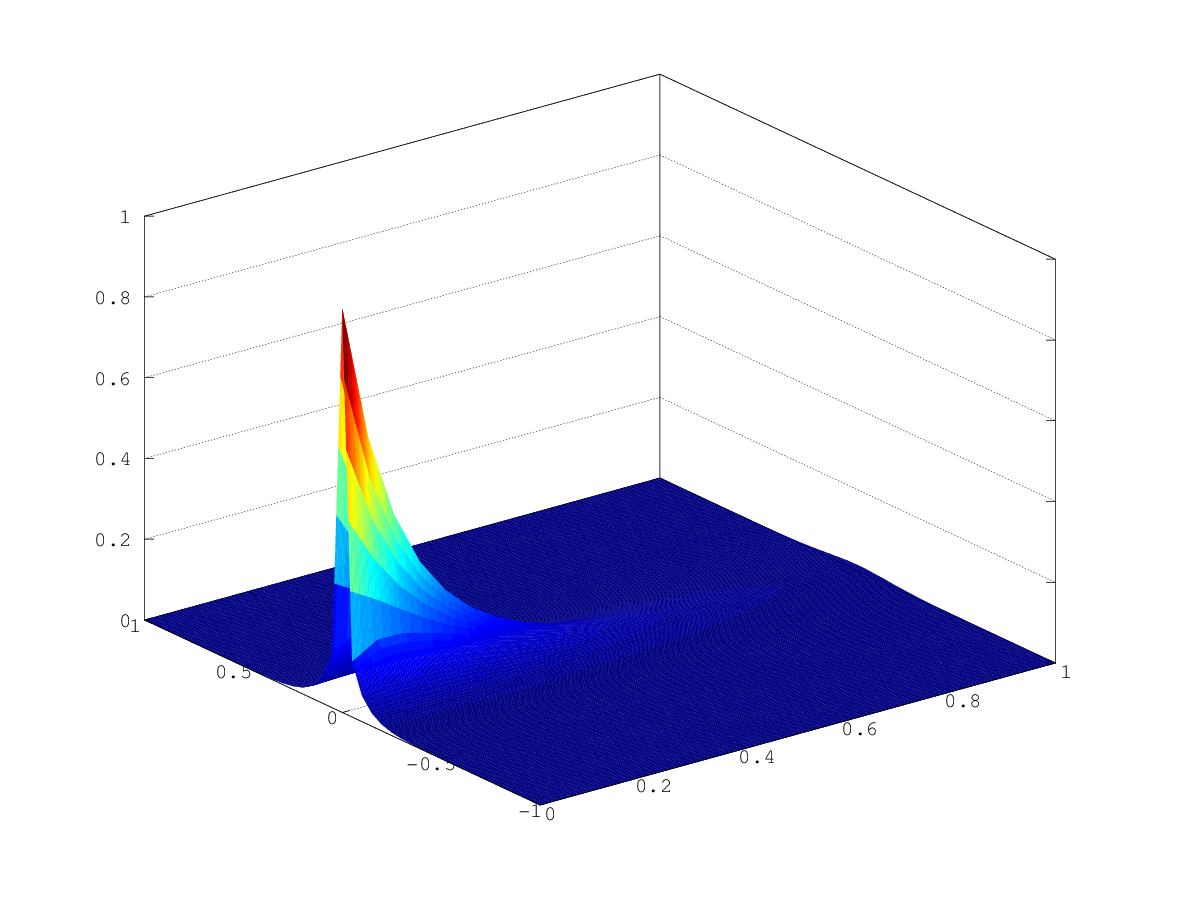}
\includegraphics[height=1.5in, trim=50 0 0 50]{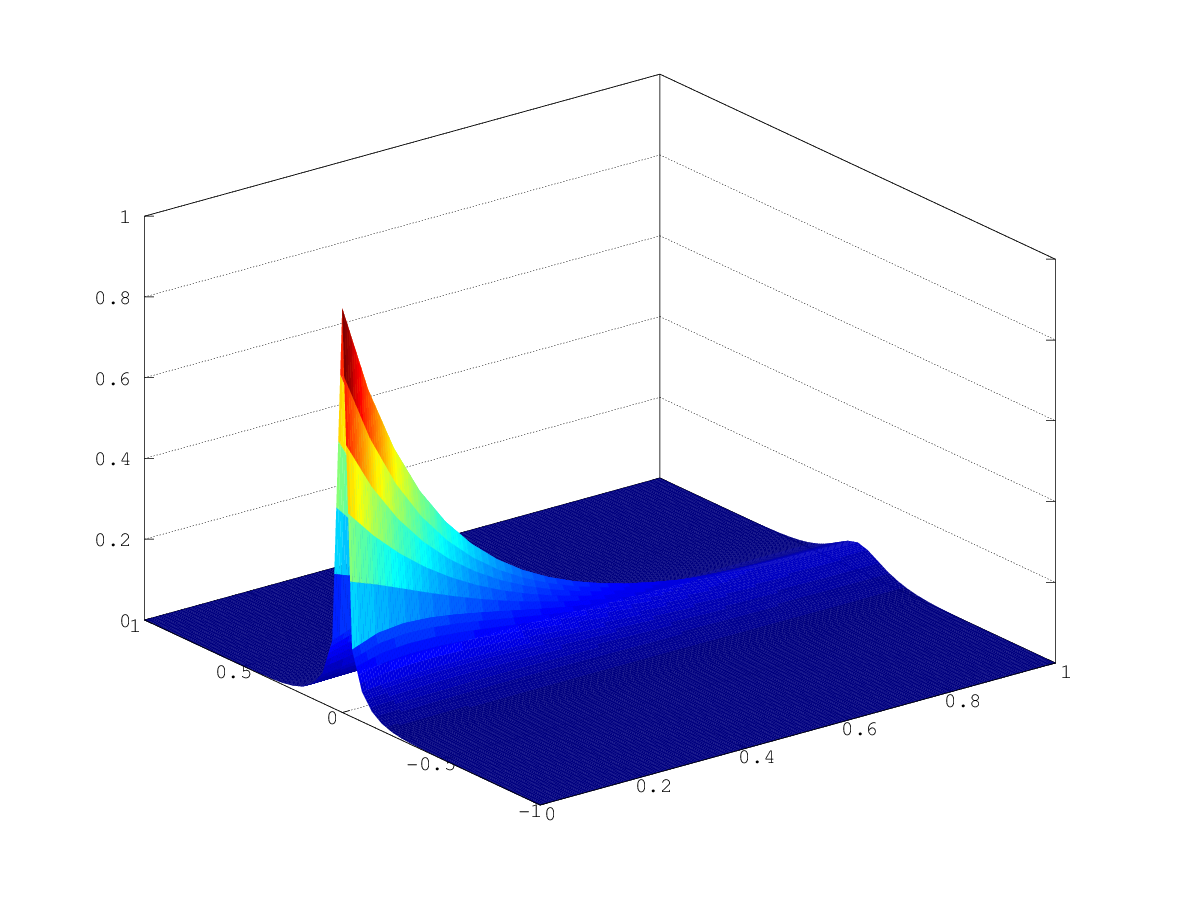}
\end{center}
\begin{center}
\includegraphics[height=1.5in, trim=65 0 0 50]{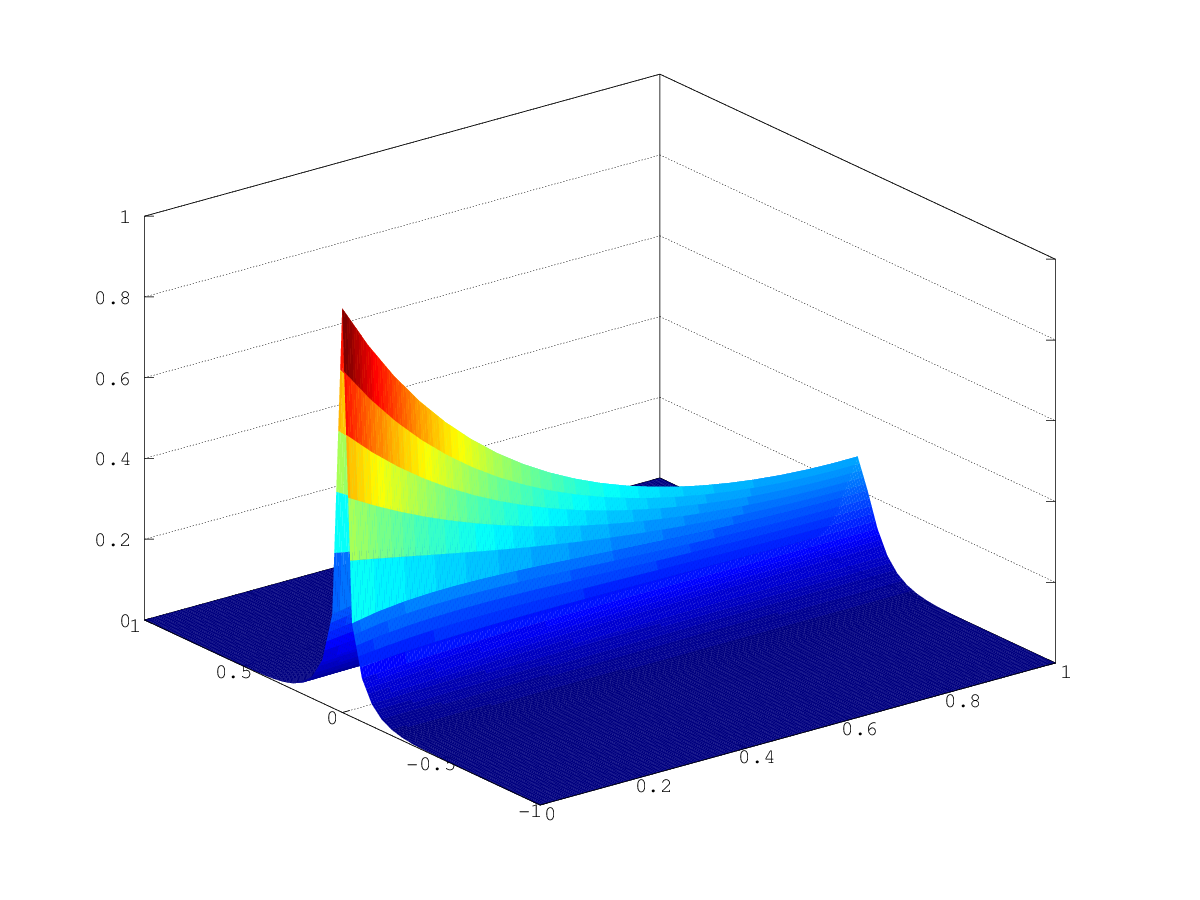}
\includegraphics[height=1.5in, trim=50 0 0 50]{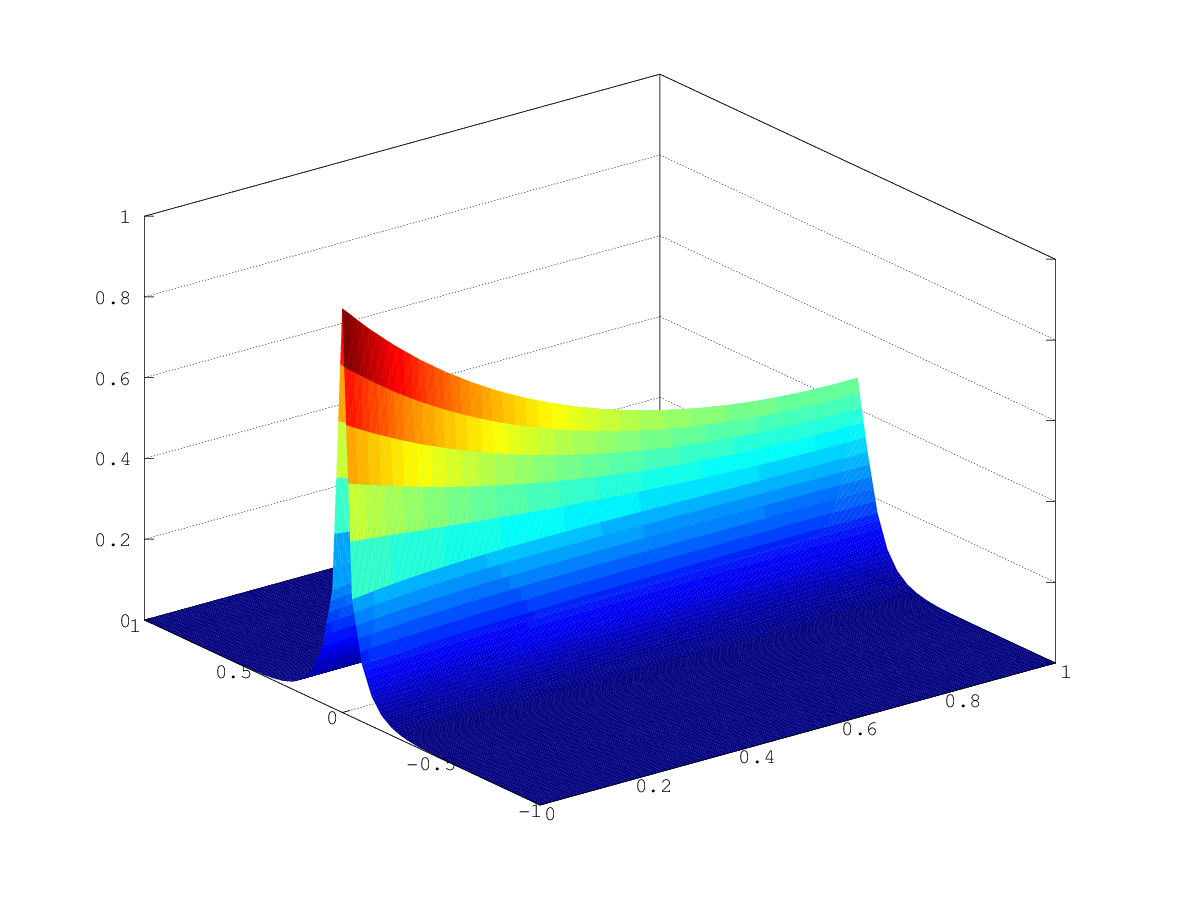}
\includegraphics[height=1.5in, trim=50 0 0 50]{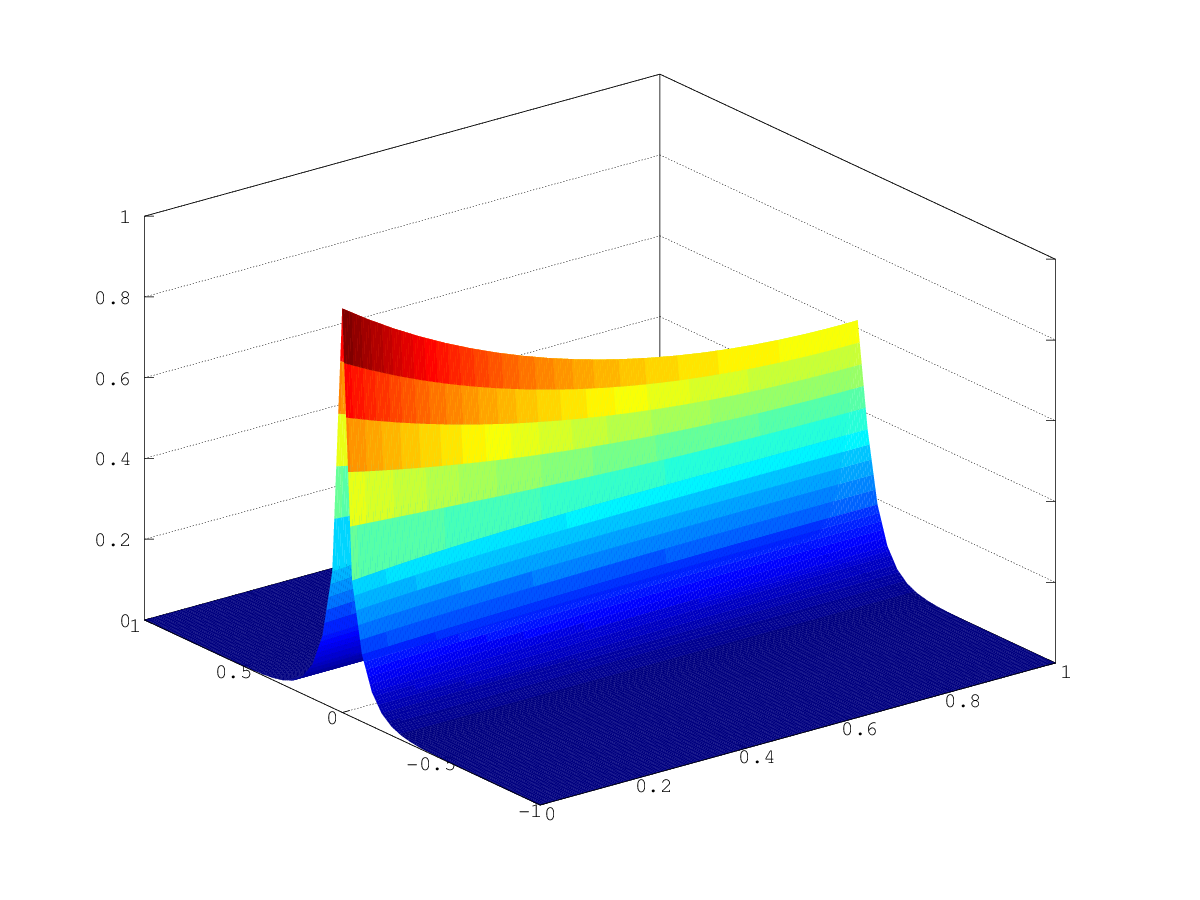}
\end{center}
\begin{center}
\includegraphics[height=1.5in, trim=65 0 0 50]{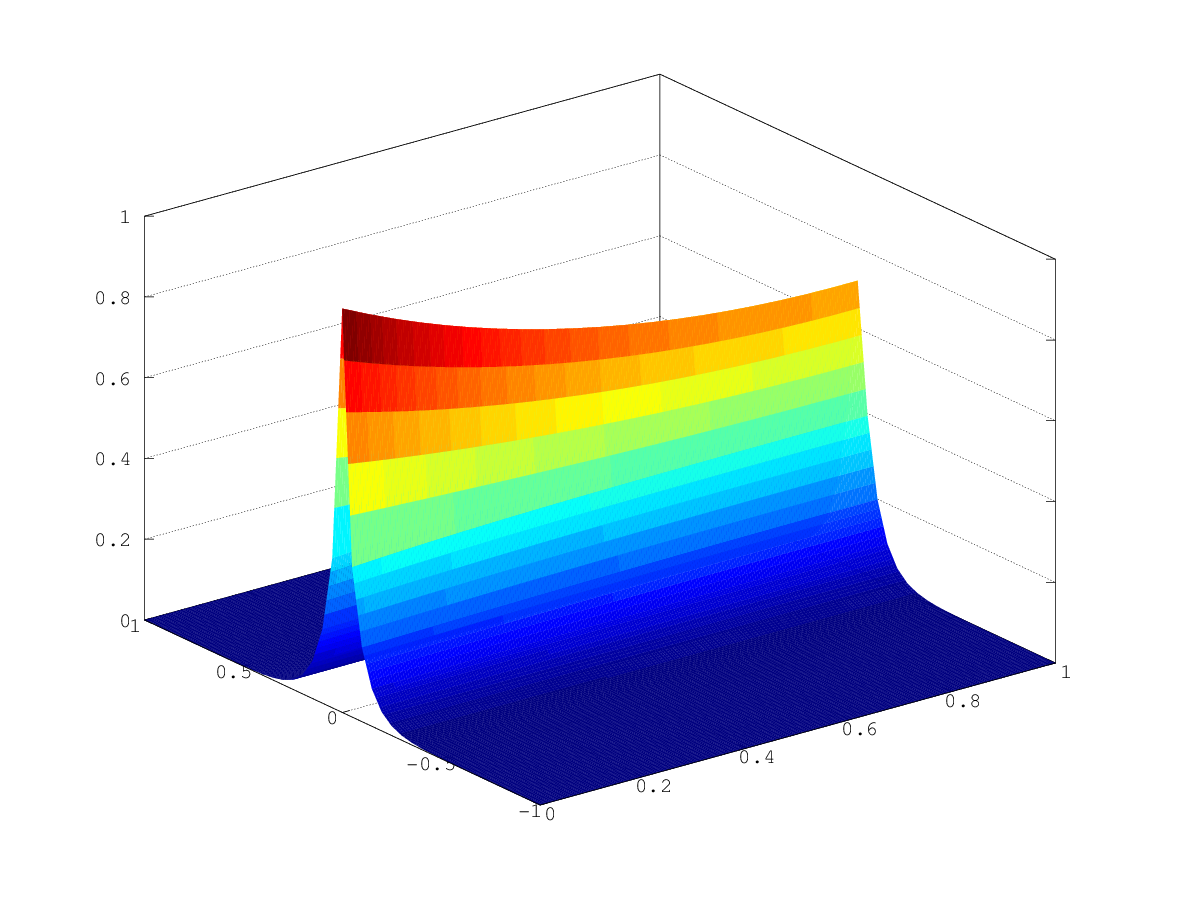}
\includegraphics[height=1.5in, trim=50 0 0 50]{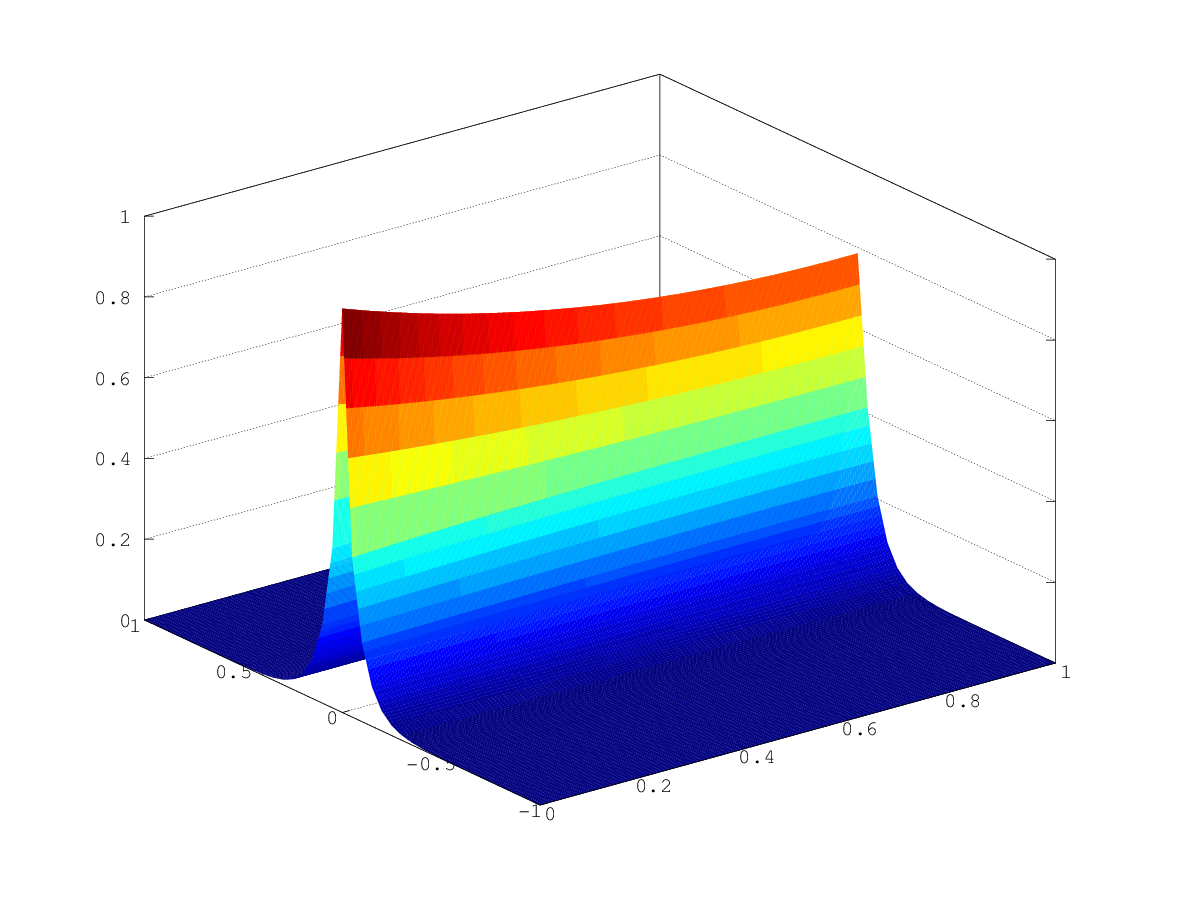}
\includegraphics[height=1.5in, trim=50 0 0 50]{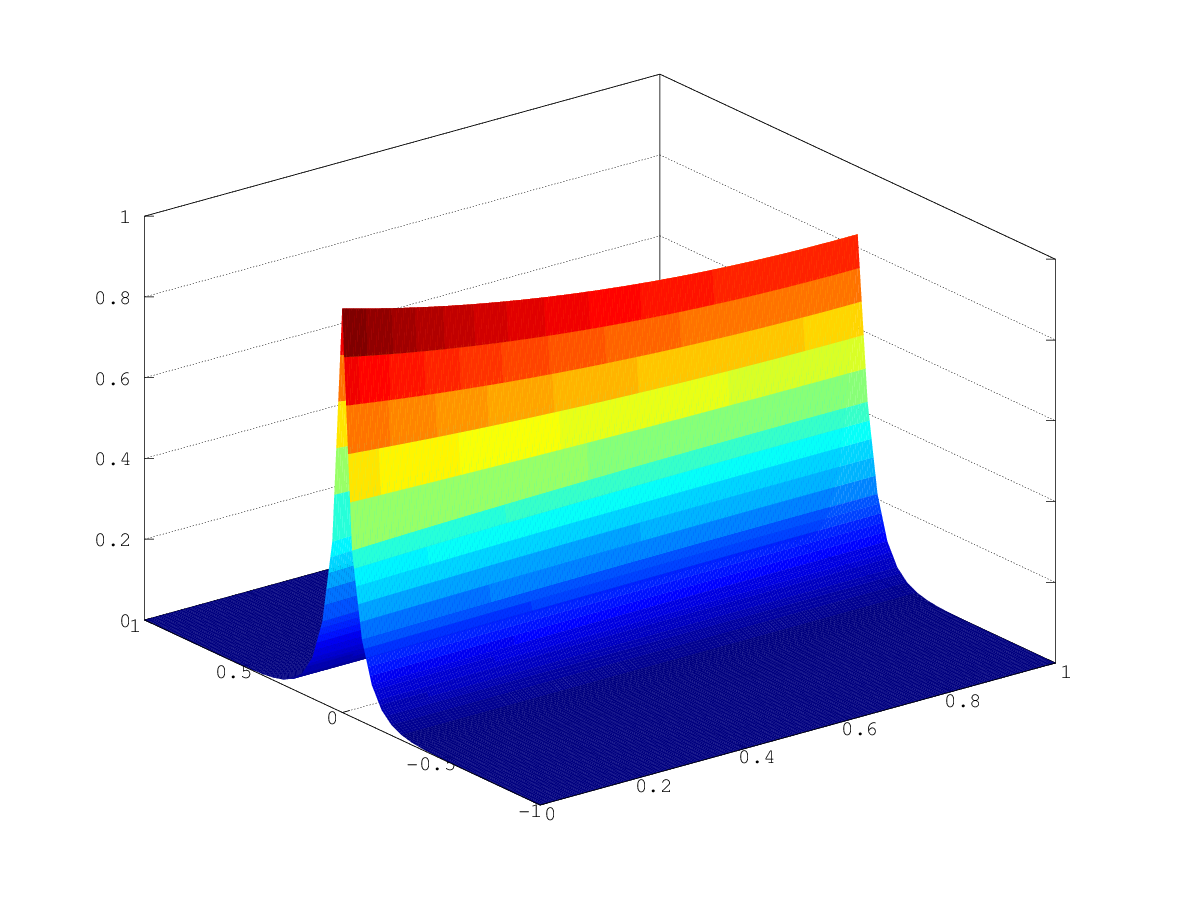}
\end{center}
\begin{center}
\includegraphics[height=1.5in, trim=65 0 0 50]{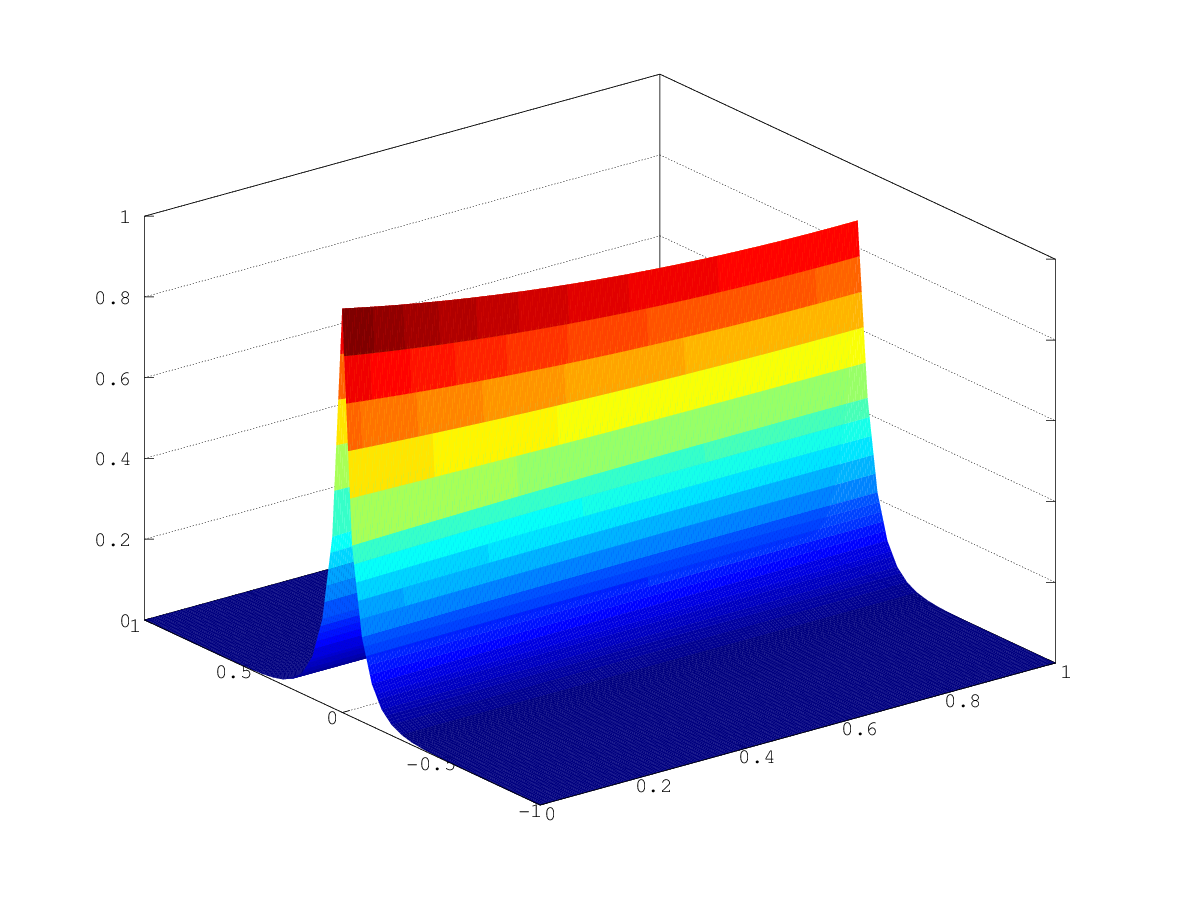}
\includegraphics[height=1.5in, trim=50 0 0 50]{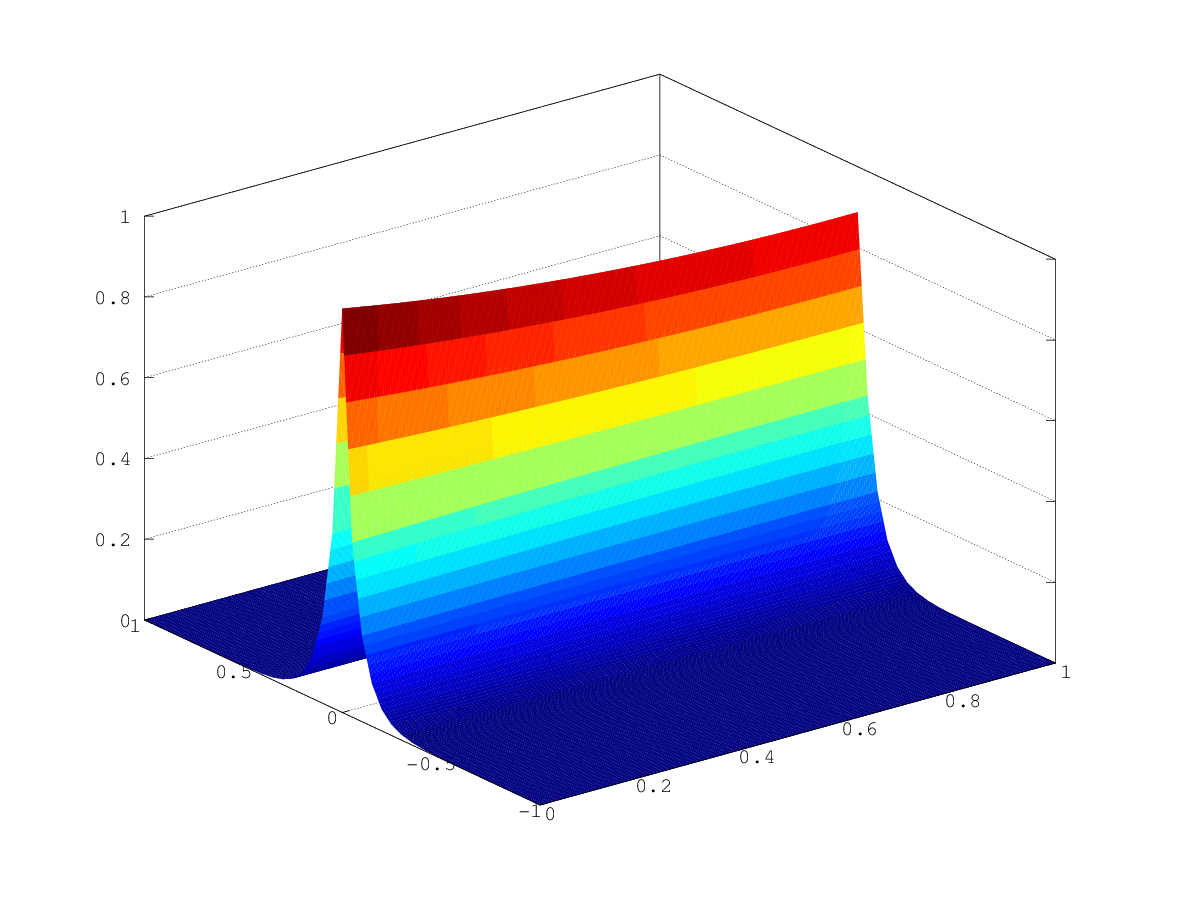}
\includegraphics[height=1.5in, trim=50 0 0 50]{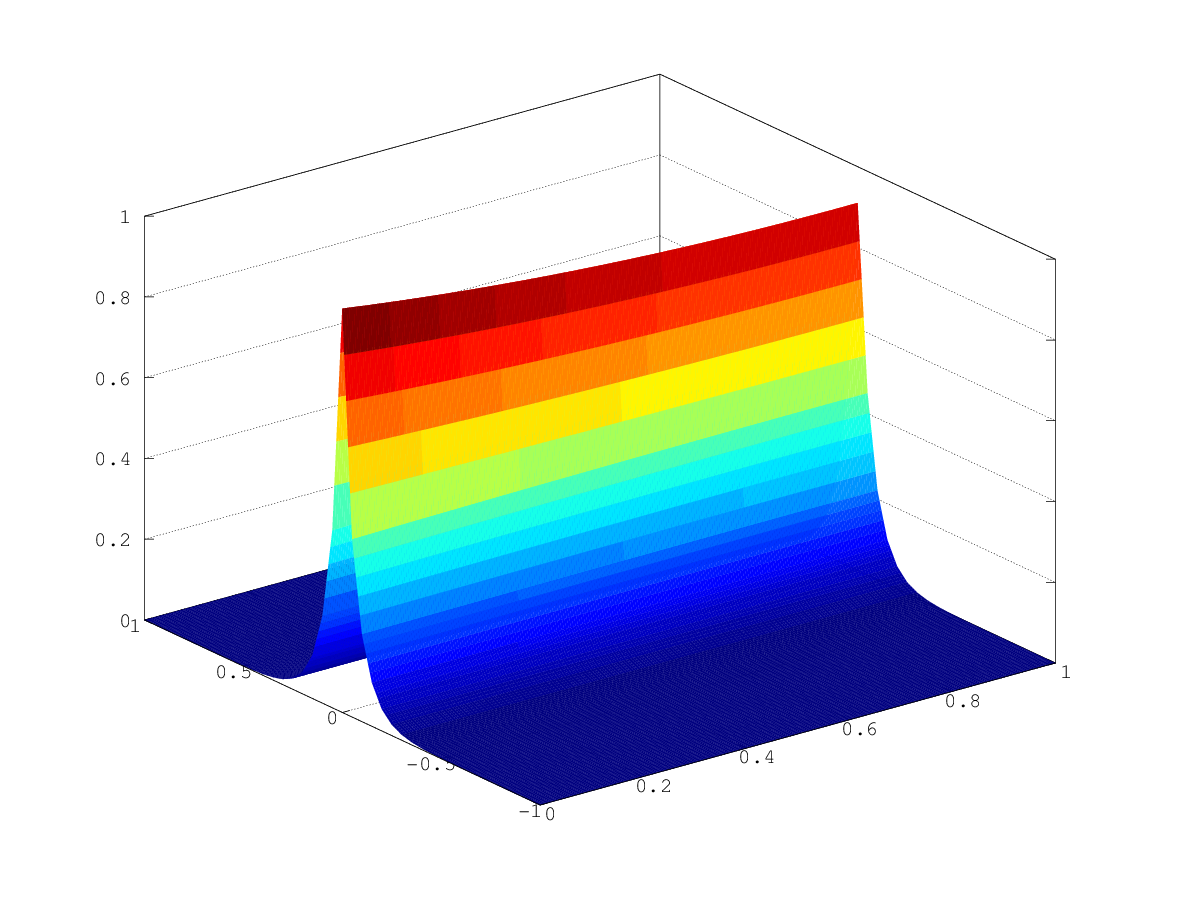}
\end{center}
\caption{Graph of $u_{1}^{h,}$ (normalised to 1) computed for the following values of parameters: $\mathbf{b}=[100,10,10,10,10],\ \mathbf{c}=[10,10,1,10,10],\ \mathbf{p}=[100,0,100,0,0],\ d=1/10$.
First row - $h\in\{1,1/3,1/5\}$, second row - $h\in\{1/10, 1/15,1/20\}$, third row - $h\in\{1/25,1/30, 1/35\}$, fourth row - $h\in\{1/40,1/45,1/50\}$. A numerical scheme based on the finite difference method was implemented using the software Octave.} 
\label{H2S:wykres}
\end{figure}

\begin{proof}
We divide the proof into two parts. In the first part we employ the technique from \cite{BocGau} to prove existence of the solution which additionally satisfies \eqref{H2S:bvp1est}. Notice that one has to use a slight modification due to the Robin boundary condition instead of the Dirichlet condition which is treated in \cite{BocGau}. In the second part of the proof, using duality technique from \cite{BreStr}, we show that the solution is unique in the $W^1_1$ class. 
\\
\newline
\textbf{Existence}\\
Observe that due to linearity of the problem \eqref{H2S:bvp1} one can assume without loss of generality that
\begin{align*}
\n{\mu_{\Omega}}_{TV}+\n{\mu_I}_{TV}\leq 1.
\end{align*}
First let us consider $\mu_{\Omega}\in L_{\infty}(\Omega), \mu_I\in L_{\infty}(\partial_1\Omega)$.
Using the Lax-Milgram lemma we obtain that the problem \eqref{H2S:bvp1} has a unique solution $u\in W^1_2(\Omega)$. We will now prove that this solution satisfies \eqref{H2S:bvp1est}. Observe that if $\phi\in W^1_{\infty}(\mathbb{R})$ is such that 
\begin{align}
\n{\phi}_{L_\infty(\mathbb{R})}\leq 1, \ y\phi(y)\geq 0, \ \phi'(y)\geq 0,
\end{align}
then testing \eqref{weaklem1} by $v=\phi(u)\in W^1_2(\Omega)$ we obtain
\bs
\eq{
\lambda\int_{\Omega}u\phi(u)&\leq 1, \label{H2S:lem1proofest1}\\
0\leq\int_{\Omega}-\phi'(u)J_{h}(u)\nabla u&\leq 1.\label{H2S:lem1proofest2}
}
\es

For $n\geq 1$ define
\begin{align}\label{H2S:varphi}
\varphi_n(y)=\begin{cases}
ny & \mbox{if } |y|<1/n\\
sgn(y) & \mbox{if }  |y|\geq 1/n	
\end{cases}.
\end{align}

Choosing in \eqref{H2S:lem1proofest1} $\phi=\varphi_n$ and taking $n \to\infty$ we obtain that
\begin{align}
\n{u}_{L_1(\Omega)}\leq 1/\lambda\leq C\label{H2S:L1est}.
\end{align} 

For $n\geq0$ define $B_n=\{x: n\leq|u(x)|\leq n+1\}$ and
\begin{align*}
\psi_n(y)=\begin{cases}
0 & \mbox{if } |y|<n\\
y-sgn(y)\cdot n & \mbox{if }  n\leq|y|\leq n+1\\
sgn(y) & \mbox{if } |y|>n+1	
\end{cases}.
\end{align*}
Choosing in \eqref{H2S:lem1proofest2} $\phi=\psi_n$ we obtain that
\begin{align}
\n{m_h(u)\mathbf{1}_{B_n}}_{L_2(\Omega)}^2=\int_{B_n}-J_h(u)\nabla u\leq1,
\end{align}
where $m_h(u)=\sqrt{|\partial_{x_1}u|^2+h^{-2}|\partial_{x_2}u|^2}$. Using H\"older's inequality with $1=p/2+p/p^*$ we have
\begin{align}
\n{m_h(u)\mathbf{1}_{B_n}}^p_{L_p(\Omega)}\leq\n{m_h(u)\mathbf{1}_{B_n}}^p_{L_2(\Omega)}|B_n|^{p/p^*}\leq |B_n|^{p/p^*}\leq C\label{H2S:Hol}.
\end{align}
Using Sobolev's inequality and \eqref{H2S:L1est} we have
\begin{align}
\n{u}_{L_{p^*}(\Omega)}\leq C(\n{m_1(u)}_{L_p(\Omega)}+\n{u}_{L_1(\Omega)})\leq C(\n{m_h(u)}_{L_p(\Omega)}+1)\label{H2S:sob}.
\end{align}
From \eqref{H2S:Hol}, H\"older inequality (for series) and \eqref{H2S:sob} we have
\begin{align*}
\n{m_h(u)}_{L_p(\Omega)}^p&=\sum_{n=0}^{N}\n{m_h(u)\mathbf{1}_{B_n}}_{L_p(\Omega)}^p+\sum_{n=N+1}^{\infty}\n{m_h(u)\mathbf{1}_{B_n}}_{L_p(\Omega)}^p\leq C(N+1)+\sum_{n=N+1}^{\infty}|B_n|^{p/p^*}\\
&\leq C(N+1)+\sum_{n=N+1}^{\infty}n^{-p}\n{u\mathbf{1}_{B_n}}_{L_{p^*}(\Omega)}^p\leq C(N+1)+\Big(\sum_{n=N+1}^{\infty}n^{-2}\Big)^{p/2}\n{u}_{L_{p^*}(\Omega)}^p\\
&=C(N+1)+A(N)^p\n{u}_{L_{p^*}(\Omega)}^p\leq C(N+1)+CA(N)^p(\n{m_h(u)}_{L_p(\Omega)}^p+1).
\end{align*}

Taking $N$ sufficiently large we obtain
\begin{align*}
\n{m_h(u)}_{L_p(\Omega)}\leq C.
\end{align*}
Finally from \eqref{H2S:sob} it follows that $\n{u}_{L_p(\Omega)}\leq C\n{u}_{L_{p^*}(\Omega)}\leq C(\n{m_h(u)}_{L_p(\Omega)}+1)\leq C$ which completes the proof of \eqref{H2S:bvp1est}.\\

The case of arbitrary Radon measures $\mu_{\Omega}, \mu_{I}$ follows by standard approximation, see \cite{BocGau} for instance.

\textbf{Uniqueness}\\
We shall use duality technique. Let $u$ be a $W^1_1$ solution to  
\bs\label{H2S:bvp1zero}
\eq{
div(J_h(u))+(\lambda+a_0)u&=0,&&x\in\Omega\\
-J_h(u)\nu&=0,&&x\in\partial_0\Omega\\
-J_h(u)\nu+a_{11}u&=0,&&x\in\partial_1\Omega.
}
\es
We intend to prove that $u\equiv0$.
First we assume additionally that $a_{11}\equiv 0$. Using \cite{Fai} we get that for every $f\in L_q(\Omega), q>2$, problem
\bs\label{H2S:bvp1dual}
\eq{
div(J_h(v))+(\lambda+a_0)v&=f, && x\in \Omega\\
-J_h(v)\nu&=0, &&x\in\partial\Omega
}
\es
has a unique solution $v\in W^2_{q}(\Omega)$. Since $q>2$ we have $W^2_q(\Omega)\subset W^1_{\infty}(\Omega)$, so that for every $w\in W^1_1(\Omega)$ we have
\begin{align*}
\int_{\Omega}[-J_h(v)\nabla w+(\lambda+a_0)vw]=\int_{\Omega}fw.
\end{align*}
Taking $w=u$ we thus get $\int_{\Omega}fu=0$ and since $f$ was arbitrary - $u\equiv0$ follows.\\
Now let us take $0\leq a_{11}\in L_{\infty}(\partial_1\Omega)$. Denote $g=-a_{11}u$. Observe that $u$ is a $W^1_1$ solution of 
\bs\label{H2S:bvp1boost}
\eq{
div(J_h(u))+(\lambda+a_0)u&=0, && x\in \Omega\\
-J_h(u)\nu&=0, &&x\in\partial_0\Omega\\
-J_h(u)\nu&=g, &&x\in \partial_1\Omega.
}
\es
As we already showed \eqref{H2S:bvp1boost} has a unique $W^1_1$ solution and, thus $u\in W^1_p(\Omega)$ for every $p<2$. In particular $g\in L_q(\partial_1\Omega)$ for every $q<\infty$. We can now use Lax-Milgram theorem to prove that \eqref{H2S:bvp1boost} has a unique $W^1_2$ solution and thus conclude that $u\in W^1_2(\Omega)$. It follows that, $u$ is also a $W^1_2$ solution of \eqref{H2S:bvp1zero}, whence $u\equiv0$. 

\end{proof}

\begin{lem}\label{H2S:bvp2lem}
Assume that $d>0, \ 0\leq a_0\in L_{\infty}(\Omega)$ and
\begin{align}\label{H2S:dominant}
a_{ij}\in L_{\infty}(\partial_1\Omega), \ a_{11}\geq|a_{21}|, a_{22}\geq|a_{12}|.
\end{align}
Then for every $h\in(0,1], \lambda>0, \mu_{\Omega}\in\mathcal{M}(\Omega), \mu_I\in\mathcal{M}(I)$ the following system
\bs\label{H2S:bvp2}
\eq{
div(J_h(u_1))+(\lambda+a_0)u_1&=\mu_{\Omega},&& x\in\Omega\label{H2S:bvp2A}\\
-d\partial^2_{x_1}u_2-a_{21}u_1+(\lambda+a_{22}) u_2&=0,&& x\in\partial_1\Omega\label{H2S:bvp2B}
}
\es
with boundary conditions
\bs\label{H2S:bvp2bcd}
\eq{
-J_h(u_1)\nu&=0,&&x\in\partial_0\Omega\\
-J_h(u_1)\nu+a_{11}u_1-a_{12}u_2&=\mu_I,&&x\in\partial_1\Omega\\
\partial_{x_1}u_2&=0,&&x\in\partial\partial_1\Omega,
}
\es
has a unique $W_1^1$ solution i.e. there exists a unique $(u_1,u_2)\in W^1_1(\Omega)\times W^1_1(\partial_1\Omega)$ such that for every $(v_1,v_2)\in W^1_{\infty}(\Omega)\times W^1_{\infty}(\partial_1\Omega)$:
\begin{align*}
\int_{\Omega}[-J_h(u_1)\nabla v_1+(\lambda+a_0)u_1v_1]+\int_{\partial_1\Omega}\Big[d\partial_{x_1}u_2\partial_{x_1}v_2+\lambda u_2v_2-\Big(M(u_1,u_2)\Big|(v_1,v_2)\Big)_{\mathbb{R}^2}\Big]\\
=\int_{\Omega}v_1d\mu_{\Omega}+\int_{\partial_1\Omega}v_1d\mu_I,
\end{align*}
where $M(u_1,u_2)=(-a_{11}u_1+a_{12}u_2,a_{21}u_1-a_{22}u_2)$.

Moreover $(u_1,u_2)\in W^1_p(\Omega)\times W^2_q(\partial_1\Omega)$ for every $1\leq p<2, 1\leq q<\infty$ and
\begin{align}\label{H2S:bvp2est}
\n{u_1}_{W^1_p(\Omega)}+h^{-1}\n{\partial_{x_2}u_1}_{L_p(\Omega)}+\n{u_2}_{W^2_q(\partial_1\Omega)}\leq C(\n{\mu_{\Omega}}_{TV}+\n{\mu_I}_{TV},
\end{align}
where $C$ depends only on $p,\lambda,d,\n{a_0}_{L_\infty(\Omega)},\n{a_{ij}}_{L_\infty(\partial_1\Omega)}$. If $\mu_{\Omega}, \mu_I ,a_{12},a_{21}\geq0$ then $u_1,u_2\geq 0$.
\end{lem}

\begin{proof}

\textbf{Existence}\\
Let us define the Hilbert spaces $X_{1/2}=W^1_2(\Omega)\times W^1_2(\partial_1\Omega), \ X_{-1/2}=X_{1/2}^*$ and an unbounded operator \\$A: X_{-1/2}\supset X_{1/2}\to X_{-1/2}$ by
\begin{align*}
&\Big<A(u_1,u_2),(v_1,v_2)\Big>_{(X_{-1/2},X_{1/2})}=\int_{\Omega}[J_h(u_1)\nabla v_1-a_0u_1v_1]
+\int_{\partial_1\Omega}\Big[-d\partial_{x_1}u_2\partial_{x_1}v_2+\Big(M(u_1,u_2)\Big|(v_1,v_2)\Big)_{\mathbb{R}^2}\Big].
\end{align*}
Due to boundedness of $a_0$ and $a_{ij}$  operator $\lambda-A$ is coercive for $\lambda$ large enough and the Lax-Milgram lemma guarantees that there is $\lambda_0>0$ such that $[\lambda_0,\infty)\subset\rho(A)$  ($\rho(A)$ denotes the resolvent set of $A$) . Because $X_{1/2}$ is compactly embedded into $X_{-1/2}$ we get that for $\lambda\in\rho(A)$ the resolvent operator $(\lambda-A)^{-1}$ is compact and thus the spectrum $\sigma(A)$ consists entirely of eigenvalues. Choose any $\lambda\in\mathbb{R}, \ \theta\in X_{-1/2}$ and $u=(u_1,u_2)\in X_{1/2}$ such that $(\lambda-A)u=\theta$. Let $\varphi_n$ be the function defined in \eqref{H2S:varphi}. Then
\begin{align*}
&\Big<\theta,(\varphi_n(u_1),\varphi_n(u_2))\Big>_{(X_{-1/2},X_{1/2})}=\Big<(\lambda-A)(u_1,u_2),(\varphi_n(u_1),\varphi_n(u_2))\Big>_{(X_{-1/2},X_{1/2})}\\
&=\int_{\Omega}[-\varphi'_n(u_1)J_h(u_1)\nabla u_1+(\lambda+a_0)u_1\varphi_n(u_1)]\\
&+\int_{\partial_1\Omega}\Big[d\varphi'_n(u_2)|\partial_{x_1}u_2|^2-\Big(M(u_1,u_2)\Big| (\varphi_n(u_1),\varphi_n(u_2))\Big)_{\mathbb{R}^2}+\lambda u_2\varphi_n(u_2)\Big]\\
&\geq\lambda\Big(\int_{\Omega}u_1\varphi_n(u_1)+\int_{\partial_1\Omega}u_2\varphi_n(u_2)\Big)-\int_{\Omega}\Big( M(u_1,u_2)\Big| (\varphi_n(u_1),\varphi_n(u_2))\Big)_{\mathbb{R}^2}.
\end{align*}
Thus taking $n\to \infty$ and using \eqref{H2S:dominant} we get
\begin{align}
\liminf_{n\to\infty}\Big<\theta,(\varphi_n(u_1),\varphi_n(u_2))\Big>_{(X_{-1/2},X_{1/2})}\geq\lambda(\n{u_1}_{L_1(\Omega)}+\n{u_2}_{L_1(\partial_1\Omega)}).\label{H2S:blah}
\end{align}
In particular it follows from \eqref{H2S:blah} that for $\lambda>0$ equation
$(\lambda-A)u=0$
does not have nontrivial solutions, whence $(0,\infty)\subset\rho(A)$.\\ Observe that when $\mu_{\Omega}, \mu_I$ are bounded functions then the distribution $\theta$ defined by
\begin{align}
\Big<\theta,(v_1,v_2)\Big>=\int_{\Omega}v_1d\mu_{\Omega}+\int_{I}v_1d\mu_{I}=\int_{\Omega}v_1\mu_{\Omega}dx+\int_Iv_1(\cdot,0)\mu_Idx_1\label{H2S:theta}
\end{align} 
belongs to $X_{1/2}^*$ thus equation $(\lambda-A)u=\theta$ has a unique solution $u=(u_1,u_2)\in X_{1/2}$ which is a solution to problem \eqref{H2S:bvp2}-\eqref{H2S:bvp2bcd}. We will now prove that $u$ satisfies \eqref{H2S:bvp2est}.  Due to linearity of \eqref{H2S:bvp2}, \eqref{H2S:bvp2bcd} we can assume, without loss of generality, that
\begin{align*}
\n{\mu_{\Omega}}_{TV}+\n{\mu_I}_{TV}\leq1.
\end{align*}
Next we prove respectively that
\begin{align}
\lambda(\n{u_1}_{L_1(\Omega)}+\n{u_2}_{L_1(\partial_1\Omega)})&\leq C,\label{H2S:uno}\\
\n{u_1}_{W^1_p(\Omega)}+h^{-1}\n{\partial_{x_2}u_1}_{L_p(\partial_1\Omega)}&\leq C,\label{H2S:dos}\\
\n{u_2}_{W^2_q(\partial_1\Omega)}&\leq C\label{H2S:tres}.
\end{align}
To get \eqref{H2S:uno} observe that from \eqref{H2S:blah} with $\theta$ given by \eqref{H2S:theta} one has
\begin{align*}
\lambda(\n{u_1}_{L_1(\Omega)}+\n{u_2}_{L_1(\partial_1\Omega)})\leq\liminf_{n\to\infty}\Big<\theta,(\varphi_n(u_1),\varphi_n(u_2))\Big>_{(X_{-1/2},X_{1/2})}\leq\n{\mu_{\Omega}}_{L_1(\Omega)}+\n{\mu_I}_{L_1(I)}\leq 1,
\end{align*}
since $|\varphi_n(y)|\leq 1$ for $y\in\mathbb{R}$. Then \eqref{H2S:dos} follows from \eqref{H2S:uno} and Lemma \ref{H2S:bvp1lem}, while \eqref{H2S:tres} follows from \eqref{H2S:bvp2B}, \eqref{H2S:dos} and the fact that for every $1\leq q <\infty$ there exists $1\leq p<2$ such that the trace operator maps $W^1_p(\Omega)$ into $L_q(\partial_1\Omega)$. To prove existence of solutions to \eqref{H2S:bvp2}, \eqref{H2S:bvp2bcd} for the case when $\mu_{\Omega}$ and $\mu_I$ are finite Radon measures one proceeds by the standard approximation technique with the use of \eqref{H2S:bvp2est}.
\\
\newline
\textbf{Uniqueness}\\
Let $(u_1,u_2)$ be a $W^1_1$ solution of problem \eqref{H2S:bvp2}, \eqref{H2S:bvp2bcd}
with $\lambda>0, \mu_{\Omega}=0 , \mu_I= 0$.


Denoting $g_1=a_{12}u_2\in L_{\infty}(I), g_2=a_{21}u_1\in L_1(I)$ we see that $u_1$ is a $W^1_1$ solution of
\bs\label{H2S:bvp2zerou1}
\eq{
div(J_h(u))+(\lambda+a_0)u&=0,&& x\in\Omega\\
-J_h(u)\nu&=0,&&x\in\partial_0\Omega\\
-J_h(u)\nu+a_{11}u&=g_1,&&x\in\partial_1\Omega
}
\es
and $u_2$ is a $W^1_1$ solution of
\bs\label{H2S:bvp2zerou2}
\eq{
-d\partial^2_{x_1}u+(\lambda+a_{22})u&=g_2,&& x\in I\\
\partial_{x_1}u&=0,&&x\in\partial I.
}
\es
Since $g_1\in  L_{\infty}(\partial_1\Omega)$ then by Lax-Milgram lemma problem \eqref{H2S:bvp2zerou1} has a $W^1_2$ solution which by Lemma \ref{H2S:bvp1lem} is unique in $W^1_1$ class. Thus $u_1$ is a $W^1_2$ solution of \eqref{H2S:bvp2zerou1} and $g_2\in L_2(I)$. From Lax-Milgram lemma we obtain that \eqref{H2S:bvp2zerou2} has a $W^1_2$ solution which due to duality technique is unique in $W_1^1$ class. Thus $u_2\in W^1_2$. Finally we observe that $(u_1,u_2)\in X_{1/2}$ is in the kernel of the operator $(\lambda-A)$ and thus $(u_1,u_2)\equiv0$.
\end{proof}

\section{Proof of Theorem \ref{H2S:SStheo1}}
\textbf{Existence}\\
Fix $1>s>1/p, \ \infty>q>1$ and for $R>0$ define
\begin{align*}
K_R=\{(v_1,v_2)\in W^s_p(\Omega)\times L_q(\partial_1\Omega): v_1,v_2\geq0, \n{v_1}_{W^s_p(\Omega)}+\n{v_2}_{L_q(\partial_1\Omega)}\leq R\}. 
\end{align*}
$K_R$ is a bounded, convex and closed subset of the Banach space $B=W^s_p(\Omega)\times L_q(\partial_1\Omega)$. For $(v_1,v_2)\in K_R$ consider problem \eqref{H2S:SS2D}-\eqref{H2S:SS2Dbcd} with $H(u_1,u_2)$ replaced by $H(v_1,v_2)$ (notice that $v_1(0,\cdot)$ is well defined as $s>1/p$) i.e.
\bs\label{H2S:SS2D'}
\eq{
div(J_h(u_1))+b_1u_1&=0,&& x\in\Omega\label{H2S:SS2DA'}\\
-d\partial^2_{x_1}u_2-c_1u_1+(b_2+c_2+k_2H(v_1,v_2))u_2&=0,&& x\in\partial_1\Omega\label{H2S:SS2DB'}
}
\es
with boundary conditions
\bs\label{H2S:SS2Dbcd'}
\eq{
-J_h(u_1)\nu&=0,&&x\in\partial_0\Omega\\
-J_h(u_1)\nu&=-(c_1+k_1H(v_1,v_2))u_1+c_2u_2+p_1\delta,&&x\in\partial_1\Omega\\
\partial_{x_1} u_2&=0,&&x\in\partial\partial_1\Omega.
}
\es

Using Lemma \ref{H2S:bvp2lem} with
\begin{align*}
\lambda&=\min\{b_1,b_2\}, \ a_{0}=b_1-\lambda, \ \mu_{\Omega}=0, \ \mu_I=p_1\delta,\\ 
a_{11}&=c_1+k_1H(v_1,v_2), &&a_{12}=c_2,\\ 
a_{21}&=c_1, &&a_{22}=b_2-\lambda+c_2+k_2H(v_1,v_2),
\end{align*}
we obtain that problem \eqref{H2S:SS2D'} has the unique solution $(u_1,u_2)=T(v_1,v_2)$ satisfying \eqref{H2S:bvp2est} with $C$ independent of $R$ (since $H$ is bounded on $\mathbb{R}_+^2$). Thus for large $R$ the nonlinear operator $T$ maps $K_R$ into itself. Since $W^1_p(\Omega)\times W^2_q(\partial_1\Omega)$ embeds compactly into $W^s_p(\Omega)\times L_q(\partial_1\Omega)$ the nonlinear operator $T$ is compact. Since $H$ is globally Lipchitz we conclude that $T$ is continuous in the topology of $B$. Thus, using Schauder fixed point theorem, $T$ has a fixed point, which additionally satisfies \eqref{H2S:estimate}.

\textbf{Uniqueness}\\
Assume that $(u_1,u_2),(v_1,v_2)$ are two $W^1_1$ solutions of \eqref{H2S:SS2D}-\eqref{H2S:SS2Dbcd}. Denoting $z_i=u_i-v_i$ for $i=1,2$ we have:

\begin{align*}
div(J_h(z_1))+b_1z_1&=0,&& x\in\Omega\\
-d\partial^2_{x_1}z_2-c_1z_1+(b_2+c_2)z_2+k_2(H(u_1,u_2)u_2-H(v_1,v_2)v_2)&=0,&& x\in\partial_1\Omega
\end{align*}
with boundary conditions
\begin{align*}
-J_h(z_1)\nu&=0,&&x\in\partial_0\Omega\\
-J_h(z_1)\nu&=-c_1z_1-k_1(H(u_1,u_2)u_1-H(v_1,v_2)v_1)+c_2z_2,&&x\in\partial_1\Omega\\
\partial_{x_1} z_2&=0,&&x\in\partial\partial_1\Omega.
\end{align*}
Define
\begin{align*}
D&=(k_1u_1+k_2u_2+b_3)(k_1v_1+k_2v_2+b_3),\\
w_i&=(u_i+v_i)/2, \ i=1,2.
\end{align*}
Notice that
\begin{align*}
u_1v_2-u_2v_1&=z_1(u_2+v_2)/2-z_2(u_1+v_1)/2=z_1w_2-z_2w_1,\\
H(u_1,u_2)u_1-H(v_1,v_2)v_1&=p_3\Big(\frac{u_1}{k_1u_1+k_2u_2+b_3}-\frac{v_1}{k_1v_1+k_2v_2+b_3}\Big)=
\frac{p_3}{D}(k_2(u_1v_2-u_2v_1)+b_3z_1)\\
&=\frac{p_3}{D}((k_2w_2+b_3)z_1-k_2w_1z_2),\\
H(u_1,u_2)u_2-H(v_1,v_2)v_2&=p_3\Big(\frac{u_2}{k_1u_1+k_2u_2+b_3}-\frac{v_2}{k_1v_1+k_2v_2+b_3}\Big)=
\frac{p_3}{D}(-k_1(u_1v_2-u_2v_1)+b_3z_2)\\
&=\frac{p_3}{D}(-k_1w_2z_1+(k_1w_1+b_3)z_2).
\end{align*}
Thus
\begin{align*}
div(J_h(z_1))+b_1z_1&=0,&& x\in\Omega\\
-d\partial^2_{x_1}z_2-(c_1+\frac{k_1k_2p_3w_2}{D})z_1+(b_2+\frac{k_2p_3b_3}{D}+c_2+\frac{k_1k_2p_3w_1}{D})z_2&=0,&& x\in\partial_1\Omega
\end{align*}
with boundary conditions
\begin{align*}
-J_h(z_1)\nu&=0,&&x\in\partial_0\Omega\\
-J_h(z_1)\nu+(\frac{k_1p_3b_3}{D}+c_1+\frac{k_1k_2p_3w_2}{D})z_1-(c_2+\frac{k_1k_2p_3w_1}{D})z_2&=0,&&x\in\partial_1\Omega\\
\partial_{x_1} z_2&=0,&&x\in\partial\partial_1\Omega.
\end{align*}

Hence, using the notation introduced in Lemma \ref{H2S:bvp2lem}, $(z_1,z_2)$ is a $W^1_1$ solution of \eqref{H2S:bvp2},\eqref{H2S:bvp2bcd} with
\begin{align*}
\lambda&=\min\{b_1,b_2\}, \ a_{0}=b_1-\lambda, \ \mu_{\Omega}=0,\ \mu_I=0\\ 
a_{11}&=\frac{k_1p_3b_3}{D}+c_1+\frac{k_1k_2p_3w_2}{D}, &&a_{12}=c_2+\frac{k_1k_2p_3w_1}{D},\\ a_{21}&=c_1+\frac{k_1k_2p_3w_2}{D}, &&a_{22}=b_2-\lambda+\frac{k_2p_3b_3}{D}+c_2+\frac{k_1k_2p_3w_1}{D}.
\end{align*}
Since the nonnegativity of $w_1, w_2$ ensures that assumption \eqref{H2S:dominant} is fulfilled we infer that $z_1=z_2=0$.\\

\section{Proof of Theorem \ref{H2S:SStheo2}}

Since the spaces $W^1_p(\Omega)$ and $W^2_q(\partial_1\Omega)$ are reflexive for $1<p<2, \ 1<q<\infty$ thus, owing to \eqref{H2S:estimate}, there exists a sequence $(h_k)_{k=1}^{\infty}\subset(0,1]$ such that $\lim_{k\to\infty}h_k=0$ and
\bs\label{H2S:conv1}
\eq{
&u_{1}^{h_k}\rightharpoonup w_1 \quad {\rm in} \quad W^1_p(\Omega),\\
&u_{2}^{h_k}\rightharpoonup w_2 \quad {\rm in} \quad W^2_q(\partial_1\Omega).
}
\es

Now we claim that
\bs\label{H2S:conv2}
\eq{
\partial_{x_2}w_1&\equiv0,\label{H2S:ein}\\
u_{1}^{h_k}(0,\cdot)&\to w_1(0,\cdot)  \quad {\rm in} \quad L_q(\partial_1\Omega),\label{H2S:zwei}\\
u_{2}^{h_k}&\to w_2  \quad\quad\quad {\rm in} \quad C(\overline{I})\label{H2S:drei}.
}
\es

Indeed \eqref{H2S:ein} comes from \eqref{H2S:estimate}. To prove \eqref{H2S:zwei} fix any $1<q<\infty$, then choose $s,p$ such that $1<p<2, 1/p<s<1, s-2/p\geq-1/q$. Then $W^1_p(\Omega)$ embeds compactly into $W^s_p(\Omega)$, the trace operator maps $W^s_p(\Omega)$ into $W^{s-1/p}_p(\partial_1\Omega)$ and the latter space embeds continuously into $L_q(\partial_1\Omega)$. Finally \eqref{H2S:drei} follows from compact embedding of $W^2_q(\partial_1\Omega)$ into  $C(\overline{\partial_1\Omega})$. 
Choose $v_1\in C^1(\overline{\Omega}), \ v_2\in C^1(\overline{\partial_1\Omega})$, then by \eqref{H2S:weak}
\begin{align*}
\int_{\Omega}[\partial_{x_1}u_{1}^{h_k}\partial_{x_1}v_1+b_1u_{1}^{h_k}v_1]+\int_{\partial_1\Omega}[d\partial_{x_1}u_{2}^{h_k}\partial_{x_1}v_2-c_1u_{1}^{h_k}v_2]=p_1v_1(0),\\
\int_{\partial_1\Omega}[c_1H(u_{1}^{h_k},u_{2}^{h_k})u_{1}^{h_k}v_1-c_2u_{2}^{h_k}v_1+(b_2+c_2H(u_{1}^{h_k},u_{2}^{h_k})v_2)]=0.
\end{align*}
Using \eqref{H2S:conv1} and \eqref{H2S:conv2} we can pass to the limit with $k\to\infty$ and identify that $(w_1,w_2)=(u_{1}^{0},u_{2}^{0})$ is a solution of \eqref{H2S:SS1D}. Finally notice that \eqref{H2S:conv} follows from \eqref{H2S:conv1} and the fact that \eqref{SS1D} has a unique solution, as was proved in \cite{Mal2}.

\section{Acknowledgement}
The author would like to express his gratitude towards his PhD supervisors Philippe Lauren\c{c}ot \\ and Dariusz Wrzosek for their constant encouragement and countless helpful remarks and towards his numerous colleagues for stimulating discussions. \\
The author was supported by the International Ph.D. Projects Programme of Foundation for Polish Science operated within the Innovative Economy Operational Programme 2007-2013 funded by European
Regional Development Fund (Ph.D. Programme: Mathematical Methods in Natural Sciences). \\
This publication has been co-financed with the European Union funds by the European Social Fund.\\
The article is supported by the NCN grant no $2012/05/N/ST1/03115$.\\
Part of this research was carried out during the author's visit to the Institut de Math{\'e}matiques de Toulouse, Universit\'e Paul Sabatier, Toulouse III.

\end{document}